\newif\ifOptLetters
\OptLettersfalse

\documentclass[11pt,letterpaper]{article}	

\usepackage{graphicx,fancyhdr,amsmath,amssymb,subfig,url}
\usepackage{amsthm}
\usepackage[numbers,square,sort,compress]{natbib}

\newtheorem{definition}{Definition}

\newtheorem{lemma}{Lemma}
\newtheorem{theorem}{Theorem}
\newtheorem{corollary}{Corollary}

\usepackage{sectsty}
\usepackage[margin=1in]{geometry}

\usepackage[algo2e]{algorithm2e}

\usepackage{accents}

\usepackage{graphicx}
\usepackage{grffile}
\usepackage{longtable}
\usepackage{booktabs}       %
\usepackage{tablefootnote}

\usepackage{physics}
\usepackage{float,color}
\usepackage{enumerate}
\usepackage{thm-restate}
\usepackage{cleveref}

\newcommand{\zeros}{\mathbf 0}

\newcommand{\dist}{\mathop{\bf dist{}}}

\newcommand{\minimize}{\mathop{\rm minimize}}
\newcommand{\maximize}{\mathop{\rm maximize}}

\def\e{e}

\newcommand{\Lag}{\mathcal{L}}

\def\N{\mathbb{N}}

\def\R{\mathbb{R}}

\providecommand{\T}{}
\renewcommand{\T}{\top}

\newcommand{\ZStar}{Z^{\star}}

\newcommand{\diam}{\mathop{\bf diam{}}}

\newcommand{\proj}{\mathop{\bf proj}}

\SetKwProg{Fn}{Function}{:}{}
\SetKwFunction{PrimalWeightUpdate}{PrimalWeightUpdate}
\SetKwFunction{AdaptiveStepOfPDHG}{AdaptiveStepOfPDHG}
\SetKwFunction{StepOfPDHG}{StepOfPDHG}
\SetKwFunction{GetRestartCandidate}{GetRestartCandidate}
\SetKwFunction{InitializePrimalWeight}{InitializePrimalWeight}
\SetKwFunction{OneStepOfPDHG}{OneStepOfPDHG}
\SetKwProg{Fn}{function}{}{end}

\newcommand{\ME}{H} %
\newcommand{\submat}{W}

\newcommand{\M}{M}

\newcommand{\hoff}{\alpha}

\newcommand{\XStar}{X^\star}
\newcommand{\YStar}{Y^\star}

\newcommand{\nrow}{m_1}
\newcommand{\ncol}{m_2}
\newcommand{\nnz}{\textbf{nnz}(A)}

\newcommand{\PaperTitle}[0]{Worst-case analysis of restarted primal-dual hybrid gradient on totally unimodular linear programs}

\newcommand{\AbstractText}[0]{We analyze restarted PDHG on totally unimodular linear programs. In particular, we show that restarted PDHG finds an $\epsilon$-optimal solution in $O( \ME \nrow^{2.5} \sqrt{\nnz} \log(\ME \ncol /\epsilon) )$ matrix-vector multiplies where $\nrow$ is the number of constraints, $\ncol$ the number of variables, $\nnz$ is the number of nonzeros in the constraint matrix, $\ME$ is the largest absolute coefficient in the right hand side or objective vector, and $\epsilon$ is the distance to optimality of the outputted solution.}

\begin{document}

\title{\PaperTitle}

\author{Oliver Hinder}
\date{}

\maketitle

\abstract{\AbstractText}

\section{Introduction}
Consider the following linear program:
\begin{subequations}\label{primal-LP}
\begin{flalign}
\minimize_{x \in \R^{\ncol}} c^\T x \\
Ax = b \\
x \ge \zeros
\end{flalign}
\end{subequations}
and its dual $\maximize_{y \in \R^{\nrow}} b^\T y$ subject to $A^\T y \le c$,
where $\nrow$ and $\ncol$ are positive integers, $x$ and $y$ are the primal and dual variables, and $A \in \R^{\nrow \times \ncol}$, $c \in \R^{\ncol}$, $b \in \R^{\nrow}$ are the problem parameters.
Traditional methods for solving linear programs such as simplex %
and interior point methods %
require linear system factorizations, which have high memory overhead and are difficult to parallelize. Recently, there has been interest in using first-order methods for solving linear programs that use matrix-vector multiplication as their key primitive \cite{lin2021admm,applegate2022faster,o2016conic}. The advantage of matrix-vector multiplication is that it can be efficiently parallelized across multiple cores or machines. Moreover, matrix-vector multiplication has low memory footprint, using minimal additional memory beyond storing the problem data. These properties make these first-order methods suitable for tackling extreme-scale problems. 

First-order methods reformulate finding a primal and dual optimal solution to the linear program \eqref{primal-LP} as solving a minimax problem:
\begin{equation}\label{eq:poi-primal-dual}
    \min_{x \ge \zeros}\max_{y \in \R^{\nrow}} \Lag(x,y) = c^\T x + b^\T y - y^\T A x
\end{equation}
and then apply methods designed for solving minimax problems such as 
primal-dual hybrid gradient (PDHG) \cite{chambolle2011first} or the alternating direction method of multipliers (ADMM) \cite{douglas1956numerical}. See \citet{applegate2021practical} and \citet{applegate2022faster} for a more exhaustive introduction and list of references.

Empirically, a promising first-order method for solving linear programs is primal-dual hybrid gradient for linear programming (PDLP) \cite{applegate2021practical}. PDLP is based on restarted PDHG, combined with several other heuristics, for example, preconditioning and adaptively choosing the primal and dual step sizes. Restarted PDHG was analyzed by \citet{applegate2022faster} on linear programs, but their convergence bounds depend on the Hoffman constant of the KKT system. This bound is difficult to interpret and is not easily computable. Recent work by \citet{lu2023geometry} and \citet{xiong2023computational} develop different, more interpretable linear convergence bounds for PDHG. However, it is unclear if these new bounds 
would be useful for analyzing totally unimodular linear programs.

In this paper, we extend the analysis of \citet{applegate2022faster} to provide an explicit complexity bound when this method is applied to totally unimodular linear programs \cite{hoffmann1956integral}. Totally unimodular linear programs are an important subclass of linear programs which, for any integer right hand side and objective coefficients, all extreme points are integer. This subclass is of particular interest to the integer programming community \cite{conforti2014integer}. It also encapsulates the minimum cost flow problem, an important subclass of linear programs, for which almost linear-time algorithms exist \cite{chen2022maximum}.

This work analyzes a general-purpose linear programming method on a specialized problem. Many papers 
perform this style of analysis for the simplex method. For example, even though the simplex method has worst-case exponential runtime on general linear programs, improved guarantees for the simplex method exist for
subclasses such as Markov decision processes \cite{ye2011simplex,post2015simplex}, minimum cost flow \cite{orlin1997polynomial,dantzig1951application} and totally unimodular linear programs \cite{kitahara2013bound,mizuno2016simplex}. In particular, \citet{kitahara2013bound} shows that the number of iterations of the
simplex method to find an exact optimal solution on a totally unimodular linear program with a nondegenerate primal is 
$\ncol \lceil \nrow \| b \|_1 \log( \ncol \| b \|_1) \rceil$. Better complexities for this problem can be achieved by interior point methods \cite{lee2014path}, although at the cost of potentially much higher memory usage.

Finally, in independent work, \citet{cole2023first} studies the performance of first-order methods for linear programming on totally unimodular linear programs.
Their work is strongly related to ours. However, there are a few important differences: (i) our result studies an algorithm with proven practical performance \cite{applegate2021practical}, (ii) their bounds depend on $\log(H)$ instead of $H$ but for a fixed $H$ our bound has better dependence on $\nrow$, $\ncol$ and $\nnz$, and (iii) their results also extend to more general $A$ matrices (those with bounded max circuit imbalance measure). Our experiments summarized in Table~\ref{table:restart-pdhg-hard-example} indicate that restarted PDHG has an iteration bound of at least $\Omega(H)$ on totally unimodular linear programs.

\begin{table}[!t]
	\centering
\begin{tabular}{lrrrrr}
	\toprule
	$H$ & $10^2$ & $10^4$ & $10^6$ \\
	\midrule
	\# iterations & $1.8\times10^2$ & $1.7 \times 10^4$ & $1.7 \times 10^6$ \\
	\bottomrule
\end{tabular}
\caption{
Total number of iterations of restarted PDHG until the distance to optimality contracts by a factor of ten on the totally unimodular linear program $\min_{x_1, x_2 \ge 0} (H-1) x_1 + H x_2$ s.t.  $x_1 + x_2 = H$. The restart length is fixed but instance-wise tuned over the grid $2^1, 2^2, \dots, 2^{21}$ (with number of total iterations of the best restart length reported). The step size $\eta$ is set to 0.5 and the primal and dual variables are initialized at the origin. %
}\label{table:restart-pdhg-hard-example}
\end{table}

\paragraph{Notation}  Let $\R$ be the set of real numbers and $\N$ be the set of natural numbers starting from one. Denote $\{ 1, \dots, m \}$ by $[ m ] $.
Let $\nnz$ be the number of nonzeros in $A$. Assume $\ncol \ge \nrow$ and
that $\abs{b_i} \le \ME$ for all $i \in [\nrow]$ and $\abs{c_j} \le \ME$ for all $j \in [\ncol]$. 
 Let $\| \cdot \|_2$ be the Euclidean norm for vectors and spectral norm for matrices.
Let $\sigma_{\min}(M) := \min_{\| v \|_2 = 1} \| M v \|_2$ be the minimum singular value of a matrix $M$, 
$Z = \{ x \in \R^{\ncol} : x \ge 0 \} \times \R^{\nrow}$,
$W_r(z) := \{ \hat{z} \in Z : \| z - \hat{z} \|_2 \le r \}$, and 
$\dist(z, Z) := \min_{\hat{z} \in Z} \| z - \hat{z} \|_2$.
Let $\XStar$ be the set of optimal primal solutions to 
\eqref{primal-LP} and $\YStar$ be the set of optimal dual solutions.
Define $\ZStar = \XStar \times \YStar$.
Let $\e_i$ be a vector containing a one in the $i$th entry and zero elsewhere.
Let $\mathbf{1}$ be a matrix or vector of ones, and $\mathbf{0}$ be a matrix or vector of zeros.
Let $\left( \cdot \right)^{+} = \max\{ \cdot, \zeros \}$
where the max operator is applied element-wise.

\paragraph{Paper outline} \Cref{sec:restarted-PDHG} provides background on 
restarted PDHG, 
\Cref{sec:hoffman-bound} provides a new Hoffman bound that 
we will find useful and
\Cref{sec:analysis-restarted-PDHG} proves the main result.

\section{Background on restarted PDHG}\label{sec:restarted-PDHG}

This section introduces concepts from \citet{applegate2022faster} that will be useful for our analysis.
For ease of exposition, we specialize PDHG to linear programming (Algorithm~\ref{alg:one-step-PDHG}).
See \citet{chambolle2011first} for the general PDHG formula.

\begin{algorithm2e}[h]
	\Fn{\OneStepOfPDHG{$z, \eta$}}{
		$x' \gets \proj_{x \ge 0}\left( x - \eta (c - A^\T y) \right)$ \;
		$y' \gets y - \eta (b - A (2 x' - x))$ \;
		\Return{$(x', y')$}
	}
	\caption{One step of PDHG on \eqref{eq:poi-primal-dual}}\label{alg:one-step-PDHG}
\end{algorithm2e}

\noindent A key concept for restarted PDHG is the \emph{normalized duality gap} \cite{applegate2022faster}
defined 
for $r > 0$ as
\[
\rho_r(z) := \frac{\max_{\hat{z} \in W_r(z)} \Lag(x,\hat{y}) - \Lag(\hat{x}, y)}{r}
\]
where for conciseness, we use the notation $(\hat{x},\hat{y}) =\hat{z}$ (this notation is used throughout the paper, i.e., we also have $(x,y) = z$),
and for completeness define $\rho_0(z) := \limsup_{r \rightarrow 0^{+}} \rho_r(z)$.
The normalized duality gap is preferable over the standard duality gap,
$\max_{z \in Z} \Lag(x,\hat{y}) - \Lag(\hat{x}, y)$, because when 
$Z$ is unbounded the standard duality gap can be infinite even when 
we are arbitrarily close to an optimal solution in both the primal and dual.
Next, we introduce the definition of a primal-dual problem being sharp (Definition~\ref{def:sharp-pd-problem}) and PDHG with adaptive restarts (Algorithm~\ref{alg:restarted-pdhg}) along with three results from \citet{applegate2022faster} that we will use in this paper.

\begin{definition}[Definition~1 of \citet{applegate2022faster}]
	\label{def:sharp-pd-problem}
	We say a primal-dual problem \eqref{eq:poi-primal-dual} is $\alpha$-sharp on 
	the set $S \subseteq Z$ if
	for all $r \in (0, \diam(S)]$ and $z \in S$ that
	$\alpha \dist(z, \ZStar) \le \rho_r(z)$.
	\end{definition}

\begin{algorithm2e}[h]
\textbf{Input:} $z^{0,0}, \tau^{0}, \beta, \eta$ \\
\For{$n=0, \dots, \infty$}{
$t \gets 0$ \\
\Repeat{[$n = 0$ and $t \ge \tau^{0}$] or $\rho_{\| \bar{z}^{n,t} - z^{n,0} \|_2}(\bar{z}^{n,t}) \le \beta \rho_{\| z^{n,0} - z^{n-1,0} \|_2}(z^{n,0})$}{
$z^{n,t+1} \gets$ \OneStepOfPDHG{$z^{n,t}, \eta$} \;
$\bar{z}^{n,t+1} \gets \frac{1}{t+1} \sum_{i=1}^{t+1} z^{n,i}$ \;
$t \gets t + 1$
}
$z^{n+1,0} \gets \bar{z}^{n,t}$\;
}
\caption{PDHG with adaptive restarts \cite[Algorithm~1]{applegate2022faster}.}\label{alg:restarted-pdhg}
\end{algorithm2e} 

\begin{lemma}\label{lem:restarted-PDHG-distance-bound}
Algorithm~\ref{alg:restarted-pdhg} for $z^{0,0} \in Z$ satisfies
$z^{n,0} \in W_{\theta \dist(z^{0,0}, \ZStar)}(z^{0,0})$ for
$\theta = 2 \sqrt{\frac{1 + \eta \| A \|_2}{1 - \eta \| A \|_2}}$.
\end{lemma}

\begin{proof}
Proposition 9 of \citet{applegate2022faster} (which uses the norm $\| z \|_{\eta A} := \| x \|_2^2 - 2 \eta x^\T A y + \| y \|_2^2$) states that 
$\| z^{n,0} - z^{0,0} \|_{\eta A} \le 2 \| z^{0,0} - z^\star \|_{\eta A}$ for any starting point
$z^{0,0} \in Z$ and optimal solution $z^\star \in \ZStar$.
Proposition 7 of \citet{applegate2022faster} states that 
$(1 -  \eta \| A \|_2 ) \| z \|_2^2 \le  \| z \|_{\eta A} \le (1 + \eta \| A \|_2 ) \| z \|_2^2$
for all $z \in Z$.
Combining these two statements 
yields
$(1 - \eta \| A \|_2) \| z^{n,0} - z^{0,0} \|_2^2 \le \| z^{n,0} - z^{0,0} \|_{\eta A}^2 \le 2^2 \| z^{0,0} - z^\star \|_{\eta A}^2	\le 2^2 (1 + \eta \| A \|_2) \| z^{0,0} - z^\star \|_2^2$
for all $z^\star \in \ZStar$.
Rearranging gives the result.
\end{proof}
	
\begin{theorem}[\citet{applegate2022faster}] \label{thm:restarted-pdhg-basic-result}
Consider the sequence $\{z^{n,0}\}_{n=0}^{\infty}$ and $\{\tau^{n} \}_{n=1}^{\infty}$ generated
by Algorithm~\ref{alg:restarted-pdhg} with $\eta \in (0, 1/ \| A \|_2)$ and $\beta \in (0,1)$. 
Suppose that there exists a set 
$S \subseteq Z$ such that $z^{n,0} \in S$ for any $n \ge 0$ and the primal-dual problem \eqref{eq:poi-primal-dual}
is $\alpha$-sharp on the set $S$. Then, for each outer iteration $n \in \N$ we have
\begin{itemize}
\item The restart length, $\tau^{n}$, is upper bounded by $t^{\star}$:
$\tau^{n} \le t^{\star} := \Big\lceil \frac{2 C (q + 2)}{\alpha \beta} \Big\rceil$
with $C := \frac{2}{\eta (1 - \eta \| A \|_2)}$ and $q := 4 \frac{1 + \eta \| A \|_2}{1 - \eta \| A \|_2}$.
\item The distance to the primal-dual optimal solution set decays linearly:
$\dist(z^{n,0},\ZStar) \le \beta^{n} \frac{t^{\star}}{\tau^{0}} \dist(z^{0,0}, \ZStar)$.
\end{itemize}
\end{theorem}
\begin{proof}
See	Theorem 2 and Corollary 2 of \citet{applegate2022faster}.
\end{proof}

\begin{lemma}\label{lem:normalized-duality-gap-LP-basic-result}
For all $R \in (0,\infty)$, $z \in Z$, and $r \in (0,R]$ with $\| z \|_2 \le R$ we have
\[
\frac{1}{2} \left\| \begin{pmatrix} \frac{1}{R} (c^\T x - b^\T y )^{+} \\
A x - b \\
(A^\T y - c)^{+} \end{pmatrix} \right\|_2 \le \rho_r(z).
\]
\end{lemma} 

\begin{proof}
This is a variant of Lemma~4 of \citet{applegate2022faster}. 
To prove the result it suffices to
combine Equation~(22) and (25) of \cite{applegate2022faster}.
\end{proof}

\section{A Hoffman bound that explicitly takes into account nonnegativity and inequality constraints}\label{sec:hoffman-bound}

Hoffman bounds guarantee how much the distance to feasibility decreases 
as the constraint violation decreases.
Typical Hoffman bounds consider a linear inequality system of the form: $K z \le k$
where $K$ is a matrix and $k$ is a vector.
For example, \cite[Theorem 4.2.]{guler1995approximations}
states that for any matrix $K$, vector $k$ and vector $z$ we have:
$\dist(z, \{ w : K w \le k \}) \le \alpha \| (K z - k)^{+} \|_2$
where
$\alpha > 0$ is the minimum singular value across all 
nonsingular submatrices of $K$.

In \citet{applegate2022faster}, the authors employ Hoffman bounds to show that 
\[ 
\left\| \begin{pmatrix} (c^\T x - b^\T y )^{+} \\
	A x - b \\
	(A^\T y - c)^{+} \end{pmatrix} \right\|_2 
\]
is bounded below by a constant times the distance to optimality.
Using \Cref{lem:normalized-duality-gap-LP-basic-result}, this establishes that $\rho_r$ is sharp.

\Cref{lem:hoff-basic-result} is a Hoffman bound that
explicitly handles both inequality and nonnegativity constraints. Similar types of Hoffman bounds 
exist in the literature (e.g., \cite{pena2021new}) but we were unable to find a bound that 
could be readily adapted to our purpose.
The proof of \Cref{lem:hoff-basic-result} is a blackbox reduction 
to standard Hoffman bounds \cite[Theorem 4.2.]{guler1995approximations}.
In contrast, \citet{applegate2022faster} treats the nonnegativity constraints
as generic inequality constraints. 
However, in this paper we take advantage of the fact that the nonnegativity constraints on 
$x$ are never violated 
(due to PDHG performing a projection).
Explicitly handling these nonnegativity constraints
improves the quality of the Hoffman constant. 
In particular, the nonsingular submatrices considered 
in the calculation are smaller (because they do not contain 
the identity block that \citep[Equation~(20)]{applegate2022faster} introduces).
If we did employ the strategy of \citet{applegate2022faster} the remainder of the paper would remain 
essentially unchanged but the iteration bound in Theorem~\ref{thm:main-result} would contain 
$\ncol + \nrow$ instead of just $\nrow$ because the nonsingular 
submatrix $G$ could be much bigger.
This alternative worst-case bound is inferior for $\ncol \gg \nrow$.

To prove Corollary~\ref{lem:hoff-basic-result} we will find the following Lemma useful.

\newcommand{\region}{U_S}
\newcommand{\MatIneq}{D}
\newcommand{\IneqRHS}{d}
\newcommand{\MatEq}{F}
\newcommand{\EqRHS}{f}
\newcommand{\polytope}{P}
\newcommand{\sol}{u}

\begin{lemma}\label{lem:linear-algebra-schur-fact}
	Define
	$M_{\lambda} := \begin{pmatrix}
		M_{11} & M_{12} \\
		\zeros & \lambda  M_{22} 
	\end{pmatrix}$
	where $M_{22}$ is a square matrix,
	and assume $M_{\lambda}$ is nonsingular for some $\lambda \in (0,\infty)$. Then $M_{11}$ and $M_{22}$ are
	nonsingular and
	\[
	\lim_{\lambda \rightarrow \infty} M_{\lambda}^{-1} = \begin{pmatrix}
		M_{11}^{-1}  & \zeros \\
		\zeros & \zeros
	\end{pmatrix} \ .
	\]
\end{lemma}

\begin{proof}
	Since $M_{\lambda}$ is nonsingular and $M_{22}$ is square we deduce that $M_{11}$ is also square.
	Also, we have for all $u \neq \zeros$ that
	$\zeros \neq M_{\lambda} \begin{pmatrix} u \\ \zeros \end{pmatrix} = M_{11} u$
	which implies that $M_{11}$ is nonsingular.
	Similarly, for all $v \neq \zeros$ we have
	$\zeros \neq M_{\lambda}^\T \begin{pmatrix} \zeros  \\ v \end{pmatrix} = \lambda M_{22}^\T v$
	which implies that $M_{22}$ is nonsingular.
	Next, the Schur complement of $M_{\lambda}$ with respect to the $\lambda  M_{22}$ block is
	$S := M_{11} - \zeros \frac{1}{\lambda} M_{11}^{-1} M_{12} = M_{11}$.
	Using the Schur complement \cite{haynsworth1968schur} we have 
	\[ 
	M_{\lambda}^{-1} = \begin{pmatrix}
		M_{11}^{-1} & \frac{1}{\lambda} M_{11}^{-1} M_{12} M_{22}^{-1} \\
		\zeros & \frac{1}{\lambda} M_{22}^{-1}
	\end{pmatrix}.
	\] 
	Taking $\lambda \rightarrow \infty$ yields the desired result.
\end{proof}

\begin{corollary}\label{lem:hoff-basic-result}
For any nonnegative integers  $m, n$ and $p$ consider matrices
$\MatIneq \in \R^{m \times n}$, $\MatEq \in \R^{p \times n}$ and 
vectors $\IneqRHS \in \R^{m}$, $\EqRHS \in \R^{p}$.
Let $\region := \{ \sol \in \R^{n} : \sol_i \ge 0, \forall i \in S \}$ for some $S \subseteq [ n ]$. Define the polytope,
$\polytope := \{ \sol \in \region : \MatIneq \sol \le \IneqRHS, \MatEq \sol = \EqRHS \}$. 
Let $\mathcal{G}$ be the set of all nonsingular submatrices of the matrix $\begin{pmatrix} \MatIneq \\
\MatEq \end{pmatrix}$ and let $\hoff = \frac{1}{\max_{G \in \mathcal{G}} \| G^{-1} \|_2}$. Then, for all $\sol \in \region$
we have
\[ 
\hoff \dist( \sol , \polytope) \le \left\| \begin{pmatrix} (\MatIneq \sol - \IneqRHS)^{+} \\
	\MatEq \sol - f \end{pmatrix} \right\|_2.
\]
\end{corollary}

\begin{proof}%
	Consider the system
	\[
	\left[ K_{\lambda} := \begin{pmatrix} 
		\MatIneq \\
		\MatEq \\ 
		-\MatEq \\
		-\lambda E
	\end{pmatrix} \right] \sol \le \begin{pmatrix} 
		\IneqRHS \\
		\EqRHS \\
		-\EqRHS \\
		\zeros \end{pmatrix} =: k
	\]
	for $\lambda \in (0,\infty)$,
	where $-\lambda E \sol \le \zeros$ corresponds to the constraint $\sol \in \region$ with each row of $E$ containing exactly one entry (i.e., a one).
	\newcommand{\subK}{Q} \
	Let $\subK$ be a nonsingular submatrix of $K_{\lambda}$ decomposed into submatrices of $\MatIneq$, $\MatEq$, $-\MatEq$ and $E$ which we call $\subK_{\MatIneq}, \subK_{\MatEq}, \subK_{-\MatEq}, \subK_{E}$ respectively such that
	\[
	\subK = \begin{pmatrix} 
		\subK_{\MatIneq} \\
		\subK_{\MatEq} \\ 
		\subK_{-\MatEq} \\
		-\lambda \subK_{E}
	\end{pmatrix}.
	\]
	As each row of $E$ contains exactly one nonzero entry, each row of $\subK_E$ must contain at most one nonzero entry.
	Since $\subK$ is nonsingular it follows that $\subK_{E}^\T v \neq \zeros$ for all $v \neq \zeros$.
	By choosing $v = \e_i$ for each row $i$, we deduce each row of $\subK_{E}$ must contain exactly one nonzero entry.
	Therefore there exists matrices $M_{11}, M_{12}$ and $M_{22}$ such that $M_{11}$ and $M_{12}$ are submatrices of $\begin{pmatrix}
		\MatIneq \\
		\MatEq \\
		-\MatEq 
	\end{pmatrix}$,  $M_{22}$ is a nonsingular square submatrix of $\subK_{E}$, and 
	\[ 
	\Pi \subK = \begin{pmatrix}
		M_{11} & M_{12} \\
		\zeros & \lambda M_{22} 
	\end{pmatrix}
	\] 
	where $\Pi$ is some permutation matrix; let $\mathcal{M}$ represent the set of all such matrices.
	Note that since $\subK$ is square and $M_{22}$ is square, we must have $M_{11}$ is square.
	
	Applying \cite[Theorem 4.2.]{guler1995approximations} gives $\hoff_{\lambda} \dist( \sol , \polytope) \le \left\| (K \sol - k)^{+} \right\|_2 $ for all $\sol \in \region$ with
	$\hoff_{\lambda} := \max_{M \in \mathcal{M}} \left\| \begin{pmatrix}
		M_{11} & M_{12} \\
		\zeros & \lambda M_{22} 
	\end{pmatrix}^{-1} \right\|_2$.
	It follows 
	by Lemma~\ref{lem:linear-algebra-schur-fact} (see \Cref{sec:lem:linear-algebra-schur-fact}) that
	\ifOptLetters
	\begin{flalign*}
		&\lim_{\lambda \rightarrow \infty} \hoff_{\lambda} = \lim_{\lambda \rightarrow \infty}  \max_{M \in \mathcal{M}} \left\| \begin{pmatrix}
		M_{11} & M_{12} \\
		\zeros & \lambda M_{22} 
	\end{pmatrix}^{-1} \right\|_2 = \\
&\max_{M \in \mathcal{M}} \lim_{\lambda \rightarrow \infty} \left\| \begin{pmatrix}
		M_{11} & M_{12} \\
		\zeros & \lambda M_{22} 
	\end{pmatrix}^{-1} \right\|_2 = \max_{M \in \mathcal{M}} \| M_{11}^{-1} \|_2.
\end{flalign*}
	\else
	\[
	\lim_{\lambda \rightarrow \infty} \hoff_{\lambda} = \lim_{\lambda \rightarrow \infty}  \max_{M \in \mathcal{M}} \left\| \begin{pmatrix}
		M_{11} & M_{12} \\
		\zeros & \lambda M_{22} 
	\end{pmatrix}^{-1} \right\|_2 = \max_{M \in \mathcal{M}} \lim_{\lambda \rightarrow \infty} \left\| \begin{pmatrix}
		M_{11} & M_{12} \\
		\zeros & \lambda M_{22} 
	\end{pmatrix}^{-1} \right\|_2 = \max_{M \in \mathcal{M}} \| M_{11}^{-1} \|_2. %
	\]
	\fi
	Recall that $M_{22}$ is square, and that $\subK$ and therefore $\Pi \subK$ is nonsingular. 
	Therefore by Lemma~\ref{lem:linear-algebra-schur-fact} we deduce $M_{11}$ is nonsingular 
	and consequently either the row $-\MatEq_{i,\cdot}$ or $\MatEq_{i,\cdot}$ appears in $M_{11}$.
	Moreover, negating rows of $M_{11}$ does not effect its maximum singular value.
	Hence, there exists $G \in \mathcal{G}$ such that $\| G^{-1} \|_2 = \| M_{11}^{-1} \|_2$.
\end{proof}

\section{Analysis of restarted PDHG on totally unimodular linear programs}\label{sec:analysis-restarted-PDHG}

This section analyzes the worst-case convergence rate of 
restarted PDHG on totally unimodular linear programs.
For completeness, we first define what it means for a linear program
to be totally unimodular.
\begin{definition}[\cite{hoffmann1956integral}]
A matrix $A$ is totally unimodular if every square submatrix is unimodular (i.e., has determinant $0$, $1$, or $-1$).
The linear program \eqref{eq:poi-primal-dual} is totally unimodular if 
$A$ is totally unimodular and the entries of $c$ and $b$ are all integers.
\end{definition}

The following standard result will be useful.

\begin{lemma}\label{lem:total-unimodularity-preserved}
	If $A \in \R^{\nrow \times \ncol}$ is a totally unimodular matrix then (i) $[A ~ \e_i]$ for any $i \in [\nrow]$, (ii) $A^\T$ and
	(iii) $\begin{pmatrix}
		A & \zeros \\
		\zeros & A^\T
	\end{pmatrix}$ are totally unimodular matrices. Moreover, (iv)
	if $A$ is a nonsingular then $(A^{-1})_{ij}\in \{-1, 0, 1 \}$ for all $i \in [\nrow]$, $j \in [\ncol]$
	and $\| A^{-1} \|_2 \le \ncol = \nrow $.
\end{lemma}
\begin{proof}
	%The proofs of (i)-(iii) appear in \citet{seymour1986decomposition}.
	The proof of (ii) follows from the fact that the determinant is
	invariant to transpose. For (i) and (iii), we use the well-known formula
	for the determinant of block matrices:
	\[
	\det\begin{pmatrix}
		P & Q \\
		\zeros & S 
	\end{pmatrix} = \det(P) \det(S).
	\]
	Thus (i) follows by setting $P = 1$, $Q$ to be the $i$th row of $A$ and $S$ the
	remaining portion (also using that the absolute value of the determinant
	is preserved by row and column permutations). Next, (iii) follows by
	setting $Q = \zeros$, $P = A$ and $S = A^\top$.
	
	To see (iv), note by Cramer's rule and Lemma~\ref{lem:total-unimodularity-preserved}.i, the inverse of a square nonsingular totally unimodular matrix has 
	all entries equal to either $-1$, $0$ or $1$, and therefore
	$\| A^{-1} \|_2 \le \| A^{-1} \|_F \leq \ncol = \nrow$.
\end{proof}

\renewcommand{\v}{v}
\newcommand{\V}{V}

The key insight of this paper is Lemma~\ref{lem:spectral-norm-inverse}, which allows us to reason about 
matrices that are nonsingular and after removing one row totally unimodular.
The proof uses Lemma~\ref{lem:total-unimodularity-preserved} 
to decompose the matrix into the sum 
of a totally unimodular matrix and a rank one component. 
The Sherman-Morrison formula is then applied to analyze the 
inverse and its spectral norm.

\newcommand{\LFrows}[0]{n}

\begin{lemma}\label{lem:spectral-norm-inverse}
Suppose the matrix $\begin{pmatrix}
	\v^\T \\
	\V
\end{pmatrix}$ is nonsingular
where $\V$ is a totally unimodular matrix with $\LFrows$ rows, and $\v$ is a vector of length $\LFrows+1$ with rational entries. 
Let $\M$  be a positive integer such that for each $i \in [\LFrows+1]$, there exists an integer $k_i$ such that $\v_i = k_i / \M$. Then 
\[
\left\| \begin{pmatrix}
	\v^\T \\
	\V
\end{pmatrix}^{-1} \right\|_2 \le \LFrows + 1 + \M \left( (\LFrows + 1)^{1.5} \| \v \|_2 + \LFrows + 1 \right).
\]
\end{lemma}

\begin{proof}
First observe that $
\begin{pmatrix}
	\v^\T \\
	\V
\end{pmatrix} = \begin{pmatrix}
e_j^\T \\
\V
\end{pmatrix} + \e_1 (\v - e_j)^\T$
where $j$ is chosen such that $\submat := \begin{pmatrix}
e_j^\T \\
\V
\end{pmatrix}$
is nonsingular. Such a $j$ exists because the rows of $\V$ are linearly independent and $\V$ has $\LFrows$ rows.
Therefore, there must exist some 
$e_j$ that is not in the span of the rows of $\V$, making the dimension of the subspace 
spanned by the rows of $\submat$ equal to $\LFrows + 1$, which implies it is nonsingular.

Note that by Lemma~\ref{lem:total-unimodularity-preserved}, $W$ is totally unimodular.
By the Sherman-Morrison formula \cite{sherman1950adjustment}:
\begin{flalign}\label{eq:apply-sherman-morrison}
\begin{pmatrix}
\v^\T \\
\V
\end{pmatrix}^{-1} = \submat^{-1} - 
\frac{\submat^{-1} \e_1 (\v - \e_j)^\T \submat^{-1}}{1 + (\v - \e_j)^\T \submat^{-1} \e_1}.
\end{flalign}
As $\submat^{-1} \e_1$ is an integer vector, and $\v_i = k_i / \M$ where $k_i$ is an integer and $\M$ is a positive integer then there exists some integer $z$ such that $1 + (\v - \e_j)^\T \submat^{-1} \e_1 = z / \M$.
Since $1 + (\v - \e_j)^\T \submat^{-1} \e_1 \neq 0$ it follows that 
\begin{flalign}\label{eq:M-inv-bound}
\abs{1 + (\v - \e_j)^\T \submat^{-1} \e_1} \ge \frac{1}{\M}.
\end{flalign}
Using \Cref{lem:total-unimodularity-preserved}.iv we have
$\| \submat^{-1} \|_2 \le \LFrows + 1$,
$\| \submat^{-1} \e_1 \v \submat^{-1} \|_2 \le 
\| \submat^{-1} \e_1 \|_2 \| \v \submat^{-1} \|_2 \le (\LFrows + 1)^{0.5} \|  \submat^{-1} \|_2 \| \v \|_2 \le (\LFrows + 1)^{0.5} (\LFrows + 1) \| \v \|_2$
and $\| \submat^{-1} \e_1 \e_j^\T \submat^{-1} \|_2 \le \| \submat^{-1} \e_1 \|_2 \| \submat^{-1} \e_j \|_2 \le \LFrows + 1$.
Therefore, by \eqref{eq:apply-sherman-morrison}
and \eqref{eq:M-inv-bound}
we get
\[
\left\| \begin{pmatrix}
\v^\T \\
\V
\end{pmatrix}^{-1} \right\|_2 \le \LFrows + 1 + \M ((\LFrows + 1)^{1.5} \| \v \|_2 + \LFrows + 1).
\]
\end{proof}

\Cref{lem:main-result} characterizes
the sharpness constant of the normalized duality gap. The 
proof uses 
\Cref{lem:normalized-duality-gap-LP-basic-result} and \Cref{lem:spectral-norm-inverse}.

\begin{lemma}\label{lem:main-result}
Let $R := \lceil 8 \nrow^{1.5} \ME \rceil$.
If \eqref{primal-LP} is a totally unimodular linear program with an optimal solution, then there exists $\alpha > 0$ such that
$\hoff = \Omega\left( \frac{1}{\nrow^{2.5} \ME}  \right)$ and $\hoff \dist(z, \ZStar) \le \rho_r(z)$
for all $z \in W_R(\zeros)$ and $r \in (0, R]$.
Moreover, there exists $z^\star \in \ZStar$ such that $\| z^\star \|_2 \le R/4$.
\end{lemma}

\begin{proof}
First, we get a bound on the norm of an optimal solution. As there is an optimal solution to the linear program there exists an optimal basic feasible solution  \cite[Chapter 2]{bertsimas1997introduction}.
Let $B$ be an optimal basis with corresponding optimal solutions $x^\star$ and $y^\star$. It follows that 
$\| x^{\star} \|_2 = \| x^\star_B \|_2 = \| A_B^{-1} b \|_2 \le \| A_B^{-1} \|_2 \| b \|_2 \le \nrow \| b \|_2 \le \nrow^{1.5} \ME$
and 
$\| y^\star \|_2 = \| (A_B^{-1})^\T c_B \|_2 \le \| (A_B^{-1})^\T \|_2 \| c_B \|_2 \le \nrow \| c \|_2 \le \nrow^{1.5} \ME$
where $\| A_{B}^{-1} \|_2 \le \nrow$
by \Cref{lem:total-unimodularity-preserved}.iv.
With $z^\star = (x^\star, y^\star)$ we conclude $\| z^\star \|_2 \le \| x^{\star} \|_2  + \| y^{\star} \|_2   \le 2 \nrow^{1.5} \ME \le R/4$.
	
Consider a square nonsingular submatrix
$\begin{pmatrix}
\v^\T \\
\V
\end{pmatrix}$
of the matrix 
\[
\begin{pmatrix}
\frac{1}{R} c^\T & - \frac{1}{R} b^\T \\
A & 0 \\
0 & A^\T
\end{pmatrix}
\]
where $\v$ is a subvector of $\begin{pmatrix} \frac{1}{R} c^\T & - \frac{1}{R} b^\T \end{pmatrix}$ and $\V$ is 
a submatrix of $\begin{pmatrix}
A & 0 \\
0 & A^\T
\end{pmatrix}$. Note that $\V$ is totally unimodular by Lemma~\ref{lem:total-unimodularity-preserved}.

We now prove $\V$ contains at most $2 \nrow$ rows. Note $A$ contains at most $\nrow$ rows by definition and the submatrix of $\V$ corresponding to $A^{\T}$ contains at
most $\nrow$ columns.
Consequently,  the submatrix of $\V$ corresponding to $A^{\T}$ 
 contains at most $\nrow$ rows; otherwise, it would be row rank-deficient rendering $\begin{pmatrix}
	\v^\T \\
	\V
\end{pmatrix}$ row rank-deficient and contradicting our assumption that $\begin{pmatrix}
\v^\T \\
\V
\end{pmatrix}$ is nonsingular.

By \Cref{lem:spectral-norm-inverse} with $\M = R = O(H \nrow^{1.5})$, $\LFrows = 2 \nrow$ and using that $\| \v \|_2 \le \frac{1}{R} (\| c \|_2 + \| b \|_2) \le  2 \ME \nrow^{0.5} / R$ we get
\begin{flalign*}
\left\| \begin{pmatrix} \v^\T \\ \V \end{pmatrix}^{-1}  \right\|_2 &\le  \LFrows + 1 + \M ( (\LFrows + 1)^{1.5} \| \v \|_2 + \LFrows + 1 ) \\
&= O\left( \nrow^{2.5} \ME  \right).
\end{flalign*}
This implies by \Cref{lem:hoff-basic-result} that
\begin{flalign*}%
\left\| \begin{pmatrix} \frac{1}{R} (c^\T x - b^\T y )^{+} \\
A x - b \\
(A^\T y - c)^{+} \end{pmatrix} \right\|_2 
\ge  \Omega\left( \frac{1}{\nrow^{2.5} \ME} \right) \dist(z, \ZStar).
\end{flalign*}
Combining this inequality with Lemma~\ref{lem:normalized-duality-gap-LP-basic-result} shows
$\hoff \dist(z, \ZStar) \le \rho_r(z)$.
\end{proof}

We are now ready to prove the main result, \Cref{thm:main-result}. Note that 
$\| A \|_2$ can be readily estimated by power iteration  with high probability 
in $\tilde{O}(1)$ matrix-vector multiplications \cite{kuczynski1992estimating}. Therefore, it is possible to 
select a step size that meets the requirements of the theorem. 

\begin{theorem}\label{thm:main-result}
If \eqref{primal-LP} is a totally unimodular linear program with an optimal solution, then 
Algorithm~\ref{alg:restarted-pdhg} starting from 
$z^{0,0} = \zeros$ with $\frac{1}{4 \| A \|_2} \le \eta \le \frac{1}{2 \| A \|_2}$ and $\tau^{0} = 1$ requires at most 
\[
O\left( \ME \nrow^{2.5} \sqrt{\nnz} \log\left( \frac{\ncol \ME}{\epsilon} \right) \right)
\] 
matrix-vector multiplications to obtain 
a point satisfying $\dist(\ZStar, z^{n,0}) \le \epsilon$.
\end{theorem}

\begin{proof}
By 	\Cref{lem:restarted-PDHG-distance-bound}, we have  $z^{n,0} \in W_{\theta \dist(z^{0,0}, \ZStar)}(\zeros)$ for $\theta = 2 \sqrt{\frac{1 + \eta \| A \|_2}{1 - \eta \| A \|_2}} \le 2 \sqrt{2} \le 4$.
By \Cref{lem:main-result},
$\rho_r(z)$ is $\Omega\left(\nrow^{-2.5} \ME^{-1}  \right) $-sharp for all $\| z \|_2 \le 4 \dist(\zeros, \ZStar)$ and $r \le 4 \dist(\zeros, \ZStar)$.
Combining this with \Cref{thm:restarted-pdhg-basic-result} (using $\tau^0 = 1$) gives  $t^{\star} = \Big\lceil \frac{2 C (q + 2)}{\alpha \beta} \big\rceil = O( \nrow^{2.5} \ME \| A \|_2)$ and $\dist(z^{n,0},\ZStar) \le \beta^{n} t^{\star} \dist(z^{0,0}, \ZStar)$.
Hence, for $n \ge  \log_{1/\beta}(t^\star / \epsilon)$ we have $\dist(z^{n,0},\ZStar) \le \epsilon$ and the total number of iterations is 
$O\left( \ME \nrow^{2.5} \| A \|_2 \log\left( \frac{\nrow \ME \| A \|_2}{\epsilon} \right) \right)$.
From this bound, we obtain the result since $\| A \|_2 \le \| A \|_F \le \sqrt{\nnz}$ and $\log\left( \frac{\nrow \ME \| A \|_2}{\epsilon} \right)  \le \log\left( \frac{\ncol^2 \ME}{\epsilon} \right) \le 2 \log\left( \frac{\ncol \ME}{\epsilon} \right) $ because $\| A \|_2 \le  \sqrt{\nnz} \le \sqrt{\nrow \ncol} \le \ncol$ where the last inequality uses the assumption that $\nrow \le \ncol$. 
\end{proof}

\section*{Acknowledgments}

The author thanks Yinyu Ye and Javier Pe\~{n}a for useful feedback.
The author was supported by a Google Research Scholar award, AFOSR grant \#FA9550-23-1-0242 and by the NSF-BSF program, under NSF grant \#2239527.

\bibliographystyle{plainnat}

\bibliography{master.bib} %

\end{document}